\documentclass[12pt]{article}
\usepackage{amssymb,amsthm,amsmath,latexsym}
\usepackage{url}
\newtheorem{thm}{Theorem}
\newtheorem{prop}[thm]{Proposition}

\newtheorem{cor}[thm]{Corollary}
\theoremstyle{remark}

\newcommand{\FF}{\mathbb{F}}

\DeclareMathOperator{\supp}{supp}

\begin{document}
\title{Binary extremal self-dual codes of length $60$ and related codes}

\author{
Masaaki Harada\thanks{
Research Center for Pure and Applied Mathematics,
Graduate School of Information Sciences,
Tohoku University, Sendai 980--8579, Japan.
email: mharada@m.tohoku.ac.jp.}
}
\date{}

\maketitle

\begin{abstract}
We give a classification of four-circulant singly even self-dual 
$[60,30,d]$ codes for $d=10$ and $12$.
These codes are used to construct extremal singly even self-dual 
$[60,30,12]$ codes with weight enumerator for which no extremal
singly even self-dual code was previously known to exist.
From extremal singly even self-dual $[60,30,12]$ codes,
we also construct extremal singly even self-dual 
$[58,29,10]$ codes with weight enumerator for which no extremal
singly even self-dual code was previously known to exist.
Finally, we give some restriction on the possible weight enumerators
of certain singly even self-dual codes
with shadow of minimum weight $1$.
\end{abstract}



\section{Introduction}\label{sec:1}
Let $C$ be a (binary) singly even self-dual code.
All codes in this note are binary.
Let $C_0$ denote the 
subcode of $C$ consisting of 
codewords having weight $\equiv0\pmod4$.
The {\em shadow} $S$ of $C$ is defined to be $C_0^\perp \setminus C$.
Shadows for self-dual codes were introduced by Conway and
Sloane~\cite{C-S}
in order to derive new upper bounds for the minimum weight of
singly even self-dual codes, and to provide
restrictions on the weight enumerators of 
singly even self-dual codes.
By considering shadows, the largest possible minimum weights of
singly even self-dual codes of lengths up to $72$ are
given in~\cite[Table~I]{C-S}.
In this note,
we say that a singly even self-dual code with
the largest possible minimum weight given in~\cite[Table~I]{C-S}
is {\em extremal}.

The possible weight enumerators of extremal singly 
even self-dual codes are given in~\cite{C-S}
for lengths up to $64$ and length $72$ (see also~\cite{GH60} for length $60$).
It is a fundamental problem to find which weight enumerators 
actually occur
for the possible weight enumerators (see~\cite{C-S}).
The possible weight enumerators of extremal singly even 
self-dual $[60,30,12]$
codes are known as follows:
\begin{eqnarray*}
W_{60,1} &=& 1 +(2555+64 \beta) y^{12} + (33600 - 384 \beta)y^{14} + \cdots,
\\
W_{60,2} &=& 1 + 3451y^{12} + 24128 y^{14} + \cdots,
\end{eqnarray*}
where $\beta$ is an integer. 
If there is an extremal singly even self-dual
$[60,30,12]$ code with weight enumerator $W_{60,1}$,
then $\beta \in \{0,1,2,\ldots,8,10\}$~\cite{HM}.
For $\beta=0,1,5,7$ and $10$,
an extremal singly even self-dual code with weight enumerator $W_{60,1}$
was found in~\cite{TJ98}, \cite{BRY}, \cite{YL},
\cite{DH03} and \cite{GH60}, respectively.
An extremal singly even self-dual code with weight enumerator $W_{60,2}$
was found in~\cite{C-S}.

One of the main aims of this note is to show the following:

\begin{prop}\label{main}
There is an extremal singly even self-dual $[60,30,12]$
code with weight enumerator $W_{60,1}$ for $\beta=2,6$.
\end{prop}

These codes are constructed from four-circulant singly even 
self-dual $[60,30,d]$ codes for $d=10$ and $12$
by considering self-dual neighbors. 
It remains to determine whether there is 
an extremal singly even self-dual $[60,30,12]$
code with weight enumerator $W_{60,1}$ for 
$\beta=3,4,8$.

The possible weight enumerators of extremal singly even 
self-dual $[58,29,10]$
codes are known as follows:
\begin{eqnarray*}
W_{58,1} &=& 1 + (165 - 2 \gamma)y^{10} + (5078 + 2 \gamma)y^{12} 
+ \cdots,
\\
W_{58,2} &=& 1 +(319 - 24 \beta - 2 \gamma)y^{10} + 
(3132 + 152 \beta + 2\gamma)y^{12} + \cdots,
\end{eqnarray*}
where $\beta,\gamma$ are integers~\cite{C-S}. 
If there is an extremal singly even self-dual
$[58,29,10]$ code with weight enumerator $W_{58,2}$,
then $\beta \in \{0,1,2\}$~\cite{HM}.
An extremal singly even self-dual code with weight enumerator
$W_{58,1}$ is known for $\gamma=55$~\cite{Tsai92-2}.
An extremal singly even self-dual code with weight enumerator
$W_{58,2}$ is known for
\begin{align*}
\beta=0 \text{ and }
& \gamma\in \{2m \mid m =0,1,5,6,8,9,10,11,13,\ldots,65,68,71,79\}, \\
\beta=1  \text{ and }
& \gamma\in \{2m \mid m =13,14,16,\ldots,58,63\}, \\
\beta=2  \text{ and }
& \gamma\in \{2m \mid m =0,16,\ldots,50,55\}
\end{align*}
(see \cite{KA58}, \cite{KYS58}, \cite{YL}).

The following proposition is one of the main results of
this note.

\begin{prop}\label{prop:58}
There is an extremal singly even self-dual $[58,29,10]$
code with weight enumerator $W_{58,2}$ for
\begin{align*}
\beta=0 \text{ and }
& \gamma\in \{2m \mid m =2,3,4,7,12\}, \\
\beta=1  \text{ and }
& \gamma\in \{2m \mid m =8,9,10,11,12,15\}, \\
\beta=2  \text{ and }
& \gamma\in \{2m \mid m =4,6,7,8,9,10,11,12,13,14,15,51,52,53,54\}.
\end{align*}
\end{prop}

These codes are constructed from extremal
singly even self-dual $[60,30,12]$ codes constructed in this note
by subtracting and their self-dual neighbors.
Finally, we give some restriction on the possible weight enumerators
of certain singly even self-dual codes
with shadow of minimum weight $1$
(Proposition~\ref{prop:restriction}).
As a consequence, it is shown that $\gamma=55$ for
the possible weight enumerator $W_{58,1}$
(Corollary~\ref{cor}).
All self-dual codes in this note are singly even.
From now on, we omit the term singly even.

All computer calculations in this note
were done with the help of {\sc Magma}~\cite{Magma}.

\section{Extremal four-circulant 
self-dual $[60,30,12]$ codes}\label{sec:4cir}

An $n \times n$ circulant matrix has the following form:
\[
\left(
\begin{array}{ccccc}
r_0&r_1&r_2& \cdots &r_{n-1} \\
r_{n-1}&r_0&r_1& \cdots &r_{n-2} \\
\vdots &\vdots & \vdots && \vdots\\
r_1&r_2&r_3& \cdots&r_0
\end{array}
\right),
\]
so that each successive row is a cyclic shift of the previous one.
Let $A$ and $B$ be $n \times n$ circulant matrices.
Let $C$ be a $[4n,2n]$ code with generator matrix of the following form:
\begin{equation} \label{eq:GM}
\left(
\begin{array}{ccc@{}c}
\quad & {\Large I_{2n}} & \quad &
\begin{array}{cc}
A & B \\
B^T & A^T
\end{array}
\end{array}
\right),
\end{equation}
where $I_n$ denotes the identity matrix of order $n$
and $A^T$ denotes the transpose of $A$.
It is easy to see that $C$ is self-dual if
$AA^T+BB^T=I_n$.
The codes with generator matrices of the form~\eqref{eq:GM}
are called {\em four-circulant}.

In this section, we give a classification of extremal
four-circulant 
self-dual $[60,30,12]$ codes.
Two codes are {\em equivalent} if one can be obtained from the other by a
permutation of coordinates.
Our exhaustive search found all distinct extremal four-circulant 
self-dual $[60,30,12]$ codes,
which must be checked further for equivalence to
complete the classification.
This was done by considering all pairs of 
$15 \times 15$ circulant matrices $A$ and $B$ satisfying
the condition that 
$AA^T+BB^T=I_{15}$, the sum of 
the weights of the first rows of $A$ and $B$
is congruent to $1 \pmod 4$
and the sum of the weights is greater than or equal to $13$.
Since a cyclic shift of the first rows gives an equivalent code,
we may assume without loss of generality that 
the last entry of the first row of $B$ is $1$.
Then our computer search shows that 
the above distinct extremal four-circulant 
self-dual $[60,30,12]$ codes are divided into
$13$ inequivalent codes.

\begin{prop}
Up to equivalence, there are $13$  extremal four-circulant 
self-dual $[60,30,12]$ codes.
\end{prop}

We denote the $13$ codes by $C_{60,i}$ $(i=1,2,\ldots,13)$.
For the $13$ codes $C_{60,i}$ $(i=1,2,\ldots,13)$,
the first rows $r_A$ (resp.\ $r_B$) of the circulant matrices $A$
(resp.\ $B$) 
in generator matrices~\eqref{eq:GM} are listed in Table~\ref{Tab:60}.
We verified that the codes $C_{60,i}$ have
weight enumerator $W_{60,1}$, where $\beta$ are also listed
in Table~\ref{Tab:60}.

\begin{table}[thb]
\caption{Extremal four-circulant self-dual $[60,30,12]$
 codes $C_{60,i}$}
\label{Tab:60}
\begin{center}
{\footnotesize
\begin{tabular}{c|c|c|c}
\noalign{\hrule height0.8pt}
Code & $r_A$&$r_B$ & $\beta$\\
\hline
$C_{60, 1}$&$(1,0,1,1,1,0,0,0,0,1,1,1,0,1,1)$&$(0,0,0,0,0,0,1,0,1,0,0,1,0,0,1)$&0\\
$C_{60, 2}$&$(0,1,1,1,1,0,0,1,1,1,1,1,1,1,0)$&$(0,1,0,0,1,0,0,0,0,0,0,1,1,1,1)$&0\\
$C_{60, 3}$&$(1,1,0,0,1,0,0,1,0,1,1,0,1,1,1)$&$(0,0,0,0,1,0,0,0,0,0,0,1,1,0,1)$&0\\
$C_{60, 4}$&$(1,1,1,0,0,1,0,1,0,1,0,0,1,1,1)$&$(1,0,1,1,1,1,0,0,0,0,0,1,1,0,1)$&0\\
$C_{60, 5}$&$(1,1,1,1,0,1,1,0,1,1,1,0,1,1,0)$&$(1,1,0,1,0,1,1,1,0,0,0,1,1,1,1)$&0\\
$C_{60, 6}$&$(1,1,1,1,1,1,1,1,0,1,1,1,1,1,0)$&$(0,1,0,1,1,1,0,0,0,0,0,1,1,1,1)$&0\\
$C_{60, 7}$&$(0,1,1,0,0,1,1,1,0,1,1,0,1,1,1)$&$(0,1,0,0,1,1,0,0,0,0,0,1,1,1,1)$&0\\
$C_{60, 8}$&$(0,0,1,1,0,0,1,0,1,1,0,0,1,1,1)$&$(0,1,0,1,1,1,1,0,0,0,0,1,1,1,1)$&0\\
$C_{60, 9}$&$(0,1,1,0,1,0,0,1,0,0,1,0,1,1,0)$&$(1,0,1,0,0,0,0,0,0,0,0,1,1,1,1)$&10\\
$C_{60,10}$&$(0,1,1,1,0,1,0,1,1,1,1,0,1,1,0)$&$(0,1,0,1,0,1,1,0,0,0,0,1,1,0,1)$&10\\
$C_{60,11}$&$(0,0,1,1,1,1,1,1,1,1,0,0,1,1,0)$&$(0,1,0,1,0,1,1,0,0,0,0,1,1,0,1)$&10\\
$C_{60,12}$&$(1,1,0,0,1,0,0,0,0,1,1,0,1,1,1)$&$(0,0,0,0,0,1,0,0,0,0,0,1,1,1,1)$&10\\
$C_{60,13}$&$(1,1,1,1,1,1,0,1,1,0,1,0,1,1,0)$&$(0,0,1,0,0,1,0,0,0,0,0,1,1,1,1)$&10\\
\noalign{\hrule height0.8pt}
\end{tabular}
}
\end{center}
\end{table}

\section{Extremal self-dual $[60,30,12]$ neighbors}
\label{sec:nei}

Two self-dual codes $C$ and $C'$ of length $n$
are said to be {\em neighbors} if $\dim(C \cap C')=n/2-1$. 
Any self-dual code of length $n$ can be reached
from any other by taking successive neighbors (see~\cite{C-S}).
It is known that a self-dual code $C$ of length $n$ has 
$2(2^{n/2-1}-1)$ self-dual neighbors.
These neighbors are constructed by finding
$2^{n/2-1}-1$ subcodes of codimension $1$ in $C$
containing the all-one vector.
A computer program written in {\sc Magma},
which was used to find self-dual neighbors,
can be obtained electronically from
\url{http://www.math.is.tohoku.ac.jp/~mharada/Paper/neighbor.txt}.
In this section, we construct extremal self-dual 
$[60,30,12]$ codes by considering self-dual neighbors.

\begin{table}[thb]
\caption{Extremal self-dual $[60,30,12]$ neighbors $D_{60,i}$}
\label{Tab:nei}
\begin{center}
{\footnotesize
\begin{tabular}{c|c|l|c}
\noalign{\hrule height0.8pt}
$C$ & $D$ & \multicolumn{1}{c|}{$\supp(x)$} & $W$\\
\hline
$D_{60,1}$&$C_{60,1}$ & $\{1,31,32,38,42,43,46,47,48,50,51,55\}$&$\beta= 0$\\
$D_{60,2}$&$C_{60,1}$ & $\{2,3,8,33,35,39,40,41,46,50,54,59\}  $&$\beta= 0$\\
$D_{60,3}$&$C_{60,1}$ & $\{4,8,9,32,42,43,48,51,53,54,56,60\}  $&$\beta= 2$\\
$D_{60,4}$&$C_{60,2}$ & $\{2,32,34,38,40,43,49,52,54,55,57,59\}$&$\beta= 0$\\
$D_{60,5}$&$C_{60,4}$ & $\{1,31,35,39,40,41,42,43,50,52,54,55\}$&$\beta= 0$\\
$D_{60,6}$&$C_{60,10}$&$\{2,32,38,41,43,49,51,52,54,55,56,60\}$& $\beta=10$\\
$D_{60,7}$&$C_{60,12}$&$\{3,7,10,32,35,36,38,46,53,55,58,60\} $& $\beta=10$\\
\noalign{\hrule height0.8pt}
\end{tabular}
}
\end{center}
\end{table}

For $i=1,2,\ldots,13$,
by finding all $2(2^{29}-1)$ self-dual neighbors of $C_{60,i}$,
we determined the equivalence classes among
extremal self-dual neighbors of $C_{60,i}$.
Our computer search shows that
the code $C_{60,i}$ has $n_i$ inequivalent
extremal self-dual neighbors, 
which are equivalent to none of the $13$ codes $C_{60,j}$, where
$n_i$ are given by
\[
n_i= \begin{cases}
    3 & \text{if }i=1, \\
    1 & \text{if }i=2,4,10,12, \\
    0 & \text{otherwise.}\\
  \end{cases}
\]
We denote the $7$ extremal self-dual codes
by $D_{60,i}$ $(i=1,2,\ldots,7)$.
These codes $C=D_{60,i}$ are constructed as
\[
\langle (D \cap \langle x \rangle^\perp), x \rangle,
\]
where $D$ and 
the support $\supp(x)$ of $x$ are listed in Table~\ref{Tab:nei}.
We verified that the codes $D_{60,i}$ have
weight enumerator $W_{60,1}$, where $W$ in Table~\ref{Tab:nei}
indicates the values $\beta$ in the weight enumerator $W_{60,1}$.
The code $D_{60,3}$ has the following weight enumerator:
\begin{align*}
&
1
+ 2683 y^{12}
+ 32832 y^{14}
+ 280017 y^{16}
+ 1719808 y^{18}
+ 7800120 y^{20}
\\ &
+ 26380032 y^{22}
+ 67167368 y^{24}
+ 130134528 y^{26}
+ 193185267 y^{28}
\\ &
+ 220336512 y^{30}
+ \cdots + y^{60}.
\end{align*}
We verified that there is no pair of equivalent codes
among the $13$ codes $C_{60,i}$ and the $7$ codes $D_{60,i}$.

We continue the search to find extremal self-dual codes
by considering self-dual neighbors.
We found all inequivalent extremal self-dual 
neighbors $E_{60,i_1}$ of $D_{60,i_2}$,
which are equivalent to none of the extremal self-dual 
codes previously obtained in this note.
For the codes 
$E_{60,i_1}=\langle (D \cap \langle x \rangle^\perp), x \rangle$,
$D$ and $\supp(x)$ are listed in Table~\ref{Tab:nei2}.
In the table, $W$ indicates the values $\beta$ in 
the weight enumerator $W_{60,1}$.
By continuing this process, we found all inequivalent 
extremal self-dual neighbors of $E_{60,i}$,
which are equivalent to none of the extremal self-dual 
codes previously obtained in this note.
Finally, we verified that there is no extremal 
self-dual neighbor of $F_{60}$, 
which are equivalent to none of the extremal self-dual 
codes previously obtained in this note.

\begin{table}[thb]
\caption{Extremal self-dual $[60,30,12]$
neighbors $E_{60,i}$ and $F_{60}$}
\label{Tab:nei2}
\begin{center}
{\footnotesize
\begin{tabular}{c|c|l|c}
\noalign{\hrule height0.8pt}
$C$ & $D$ & \multicolumn{1}{c|}{$\supp(x)$} & $W$\\
\hline
$E_{60,1}$& $D_{60,2}$ & $\{2,3,6,31,32,37,39,40,46,47,54,57\}$& $\beta=0$\\
$E_{60,2}$& $D_{60,6}$ & $\{1,2,5,7,8,40,43,46,47,50,51,60\}$& $\beta=10$\\
$E_{60,3}$& $D_{60,6}$ & $\{1,4,5,8,36,38,39,40,48,53,55,60\}$& $\beta=10$\\
$E_{60,4}$& $D_{60,6}$ & $\{3,32,33,34,37,46,48,52,56,57,58,60\}$& $\beta=10$\\
\hline
$F_{60}$& $E_{60,3}$ & $\{1,2,5,35,37,40,45,49,50,55,57,59\}$& $\beta=10$\\
\noalign{\hrule height0.8pt}
\end{tabular}
}
\end{center}
\end{table}

\section{Extremal four-circulant 
self-dual $[60,30,10]$ codes and self-dual neighbors}\label{sec:d10}

Using an approach similar to that given in Section~\ref{sec:4cir},
our exhaustive search found all distinct 
four-circulant self-dual $[60,30,10]$ codes.
Then our computer search shows that 
the distinct four-circulant 
self-dual $[60,30,10]$ codes are divided into
$113$ inequivalent codes.

\begin{prop}
Up to equivalence, there are $113$ 
four-circulant self-dual $[60,30,10]$ codes.
\end{prop}

We denote the $113$ codes by $G_{60,i}$ $(i=1,2,\ldots,113)$.
For the $13$ codes $G_{60,i}$ $(i=1,2,\ldots,13)$,
the first rows $r_A$ (resp.\ $r_B$) of the circulant matrices 
$A$ (resp.\ $B$)
in generator matrices~\eqref{eq:GM} are listed in Table~\ref{Tab:d10}.
The first rows for the all codes can be obtained from
\url{http://www.math.is.tohoku.ac.jp/~mharada/Paper/60-4cir-d10.txt}.

\begin{table}[thb]
\caption{Four-circulant self-dual $[60,30,10]$
 codes $G_{60,i}$}
\label{Tab:d10}
\begin{center}
{\footnotesize
\begin{tabular}{c|c|c}
\noalign{\hrule height0.8pt}
Code & $r_A$&$r_B$ \\
\hline
$G_{60, 1}$& $(0,1,1,1,0,1,0,0,0,0,0,0,0,0,1)$&$(1,1,0,0,1,0,0,1,1,1,0,1,0,0,1)$\\
$G_{60, 2}$& $(0,0,0,1,0,0,0,1,1,1,1,0,0,0,1)$&$(0,0,0,0,0,1,0,1,1,1,0,1,0,1,1)$\\
$G_{60, 3}$& $(1,1,1,1,1,0,0,1,0,1,0,1,0,0,0)$&$(0,1,1,0,1,1,0,1,1,1,0,1,0,0,1)$\\
$G_{60, 4}$& $(1,1,1,1,1,1,0,0,1,0,1,0,1,0,0)$&$(1,1,1,0,1,1,1,1,1,1,0,1,0,1,1)$\\
$G_{60, 5}$& $(0,0,0,1,0,0,0,1,1,1,1,0,0,0,1)$&$(0,1,1,0,1,1,1,1,1,1,0,1,0,1,1)$\\
$G_{60, 6}$& $(0,0,1,1,0,1,1,0,1,0,0,1,0,0,0)$&$(1,1,1,1,1,0,1,1,1,1,0,1,0,0,1)$\\
$G_{60, 7}$& $(0,0,1,1,0,0,1,1,0,0,1,1,0,0,0)$&$(0,0,1,0,1,0,1,0,1,1,0,1,0,0,1)$\\
$G_{60, 8}$& $(0,0,0,0,0,0,1,1,1,1,0,1,0,0,0)$&$(0,1,1,1,0,1,0,0,0,1,0,1,0,1,1)$\\
$G_{60, 9}$& $(0,1,1,0,1,0,0,0,0,0,0,0,0,0,1)$&$(1,0,1,1,1,1,0,0,0,1,0,1,0,1,1)$\\
$G_{60,10}$& $(0,0,0,1,1,0,1,0,1,0,1,1,0,0,0)$&$(1,0,1,0,0,0,1,0,1,1,0,1,0,0,1)$\\
$G_{60,11}$& $(1,0,1,0,1,0,1,1,1,0,0,0,0,0,1)$&$(0,0,1,0,0,0,1,0,1,1,0,1,0,0,1)$\\
$G_{60,12}$& $(1,1,1,1,1,0,0,1,0,1,0,1,0,0,1)$&$(1,0,1,0,0,1,0,1,1,1,0,1,0,0,1)$\\
$G_{60,13}$& $(0,0,1,0,0,0,0,0,0,1,0,1,0,0,0)$&$(0,0,0,0,0,0,1,1,0,1,0,1,0,1,1)$\\
\noalign{\hrule height0.8pt}
\end{tabular}
}
\end{center}
\end{table}

In addition, we found extremal self-dual
$[60,30,12]$ codes by considering self-dual neighbors of $G_{62,i}$
$(i=1,2,\ldots,113)$.
Using a method similar to that given in~\cite{CHK64},
we completed the classification of 
extremal self-dual
$[60,30,12]$ neighbors of $G_{62,i}$
$(i=1,2,\ldots,113)$.
Our computer search shows that there is an extremal 
self-dual $[60,30,12]$ neighbor $H_{60,i}$
for $i=1,2,\ldots,13$
and that there is no extremal
self-dual $[60,30,12]$ neighbor for $i=14,15,\ldots,113$.
The codes $H_{60,i}$ are constructed as
$\langle (D \cap \langle x \rangle^\perp), x \rangle$,
where $D$ and $\supp(x)$ are listed in Table~\ref{Tab:d10nei}
and $W$ indicates the values $\beta$ in
the weight enumerator $W_{60,1}$.
We verified that there are
the following equivalent codes among
$C_{60,i_1}, D_{60,i_2}, E_{60,i_3}, F_{60}, H_{60,i_4}$:
\begin{align*}
&
H_{60,2} \cong C_{60,4},
H_{60,5} \cong C_{60,1},
H_{60,6} \cong C_{60,3},
H_{60,7} \cong C_{60,8},
H_{60,8} \cong C_{60,7},
\\&
H_{60,9} \cong C_{60,2},
H_{60,11}\cong H_{60,3},
H_{60,12}\cong H_{60,10},
H_{60,13}\cong H_{60,4},
H_{60,10}\cong D_{60,2},
\end{align*}
where $C \cong D$ means that $C$ and $D$ are equivalent.

\begin{table}[thb]
\caption{Extremal self-dual $[60,30,12]$ neighbors 
$H_{60,i}$, $J_{60,i}$, $K_{60,i}$ and $L_{60,i}$}
\label{Tab:d10nei}
\begin{center}
{\footnotesize
\begin{tabular}{c|c|l|c}
\noalign{\hrule height0.8pt}
$C$ & $D$ & \multicolumn{1}{c|}{$\supp(x)$} & $W$\\
\hline
$H_{60,1}$& $G_{60,1}$ & $\{4,30,32,36,37,40,42,43,47,51,52,54,57,58\}$& $\beta=10$\\
$H_{60,2}$& $G_{60,2}$ & $\{1,2,4,32,36,39,40,42,49,50,55,56,57,58\}$ & $\beta=0$\\
$H_{60,3}$& $G_{60,3}$ & $\{1,2,32,33,34,35,37,44,45,48,51,54,55,59\}$ & $\beta=0$\\
$H_{60,4}$& $G_{60,4}$ & $\{2,30,32,34,36,37,40,44,48,52,55,56,58,60\}$ & $\beta=0$\\
$H_{60,5}$& $G_{60,5}$ & $\{1,31,32,33,35,36,37,42,43,44,46,49,56,58\}$ & $\beta=0$\\
$H_{60,6}$& $G_{60,6}$ & $\{1,31,32,35,38,40,41,43,44,45,51,56,58,59\}$ & $\beta=0$\\
$H_{60,7}$& $G_{60,7}$ & $\{1,3,5,32,37,38,41,42,50,51,53,56,57,59\}$ & $\beta=0$\\
$H_{60,8}$& $G_{60,8}$ & $\{1,2,3,5,6,33,34,35,36,38,42,55,56,57\}$ & $\beta=0$\\
$H_{60,9}$& $G_{60,9}$ & $\{1,2,6,33,35,39,41,42,43,44,46,47,55,56\}$ & $\beta=0$\\
$H_{60,10}$& $G_{60,10}$ & $\{2,30,32,33,36,41,43,44,48,49,51,56,59,60\}$ & $\beta=0$\\
$H_{60,11}$& $G_{60,11}$ & $\{2,3,32,39,40,42,43,44,45,48,51,52,56,60\}$ & $\beta=0$\\
$H_{60,12}$& $G_{60,12}$ & $\{2,3,32,33,34,35,36,38,45,47,51,55,56,60\}$ & $\beta=0$\\
$H_{60,13}$& $G_{60,13}$ & $\{1,30,35,37,41,43,44,45,46,47,48,50,54,56\}$ & $\beta=0$\\
\hline
$J_{60,1}$& $H_{60,1}$ & $\{4,6,36,41,43,48,49,51,53,55,56,59\}$ & $\beta=10$\\
$J_{60,2}$& $H_{60,3}$ & $\{1,3,4,30,32,36,37,38,51,52,55,56\}$ & $\beta=0$\\
$J_{60,3}$& $H_{60,3}$ & $\{1,31,34,36,39,40,41,44,45,55,56,59\}$ & $\beta=0$\\
$J_{60,4}$& $H_{60,4}$ & $\{2,5,33,34,37,38,39,42,44,50,57,59\}$ & $\beta=0$\\
$J_{60,5}$& $H_{60,4}$ & $\{3,6,7,30,34,36,39,42,48,50,53,58\}$ & $\beta=6$\\
\hline
$K_{60,1}$& $J_{60,1}$ & $\{1,3,4,6,36,40,41,46,47,51,54,57\}$ & $\beta=10$\\
$K_{60,2}$& $J_{60,4}$ & $\{5,7,34,36,37,40,41,42,47,49,50,60\}$ & $\beta=6$\\
\hline
$L_{60,1}$& $K_{60,1}$ & $\{2,3,32,33,37,38,41,46,51,52,57,58\}$ & $\beta=10$\\
$L_{60,2}$& $K_{60,2}$ & $\{2,3,32,34,37,38,41,44,47,51,53,58\}$ & $\beta=6$\\
\noalign{\hrule height0.8pt}
\end{tabular}
}
\end{center}
\end{table}

Similar to Section~\ref{sec:nei},
by continuing this process, we completed a classification of
extremal self-dual neighbors
$J_{60,i}$ (resp.\  $K_{60,i}$, $L_{60,i}$), 
which are equivalent to none of the extremal self-dual 
codes previously obtained in this note, 
of $H_{60,j}$ (resp.\ $J_{60,j}$, $K_{60,j}$).
Finally, 
we verified that there is no extremal 
self-dual neighbor of $L_{60,i}$ $(i=1,2)$,
which are equivalent to none of the $37$ codes
in Tables~\ref{Tab:60}, \ref{Tab:nei}, 
\ref{Tab:nei2} and \ref{Tab:d10nei}.
We remark that there is no pair of equivalent codes
among the following $37$ codes:
\begin{align*}
&
C_{60,i}\ (i=1,2,\ldots,13),
D_{60,i}\ (i=1,2,\ldots,7),
E_{60,i}\ (i=1,2,3,4),
F_{60},
\\&
H_{60,i}\ (i=1,3,4),
J_{60,i}\ (i=1,2,3,4,5),
K_{60,i}\ (i=1,2),
L_{60,i}\ (i=1,2).
\end{align*}


The codes $D_{60,3}$ and $J_{60,5}$
(see Tables~\ref{Tab:nei} and \ref{Tab:d10nei})
establish Proposition~\ref{main}.
The code $J_{60,5}$ has the following weight enumerator:
\begin{align*}
&
1
+       2939 y^{12}
+      31296 y^{14}
+     282321 y^{16}
+    1723904 y^{18}
+    7784760 y^{20}
\\ &
+   26386176 y^{22}
+   67197064 y^{24}
+  130097664 y^{26}
+  193168371 y^{28}
\\ &
+  220392832 y^{30}
+ \cdots + y^{60}.
\end{align*}

\section{Extremal self-dual $[58,29,10]$ codes}

An extremal self-dual $[60,30,12]$ code
gives an extremal self-dual $[58,29,10]$
code by subtracting two coordinates.
We found all the extremal self-dual $[58,29,10]$
codes by subtracting from the $37$ inequivalent
extremal self-dual $[60,30,12]$ codes
given in Sections~\ref{sec:4cir},
\ref{sec:nei} and \ref{sec:d10}.
The only extremal self-dual 
$[60,30,12]$ code $D_{60,3}$ gives 
$18$ extremal self-dual $[58,29,10]$ codes $C_{58,i}$
$(i=1,2,\ldots,18)$
with weight enumerator for which no extremal
self-dual code was previously known to exist.
More precisely, the codes by subtracting $i$ and $j$
have weight enumerator $W_{58,2}$ for
$\beta=2$ and $\gamma=104$, where
$(i,j)$ are listed in Table~\ref{Tab:58}.
We verified that there are
the following equivalent codes:
\begin{align*}
&
C_{58, 1} \cong
C_{58, i}\ (i= 2, 4, 5, 7, 8,11,12,14,15,17,18),
\\&
C_{58, 3} \cong C_{58, i}\ (i=  6, 9,10,13,16),
\end{align*}
where $C_{58,1}$ and $C_{58,3}$ are inequivalent.

\begin{table}[thb]
\caption{Extremal self-dual $[58,29,10]$ codes $C_{58,i}$}
\label{Tab:58}
\begin{center}
{\footnotesize
\begin{tabular}{c|c|c|c|c|c}
\noalign{\hrule height0.8pt}
Code & $(i,j)$ & Code & $(i,j)$ & Code & $(i,j)$ \\
\hline
$C_{58, 1}$ & $( 2, 36)$ & $C_{58, 7}$ & $(12, 31)$ & $C_{58,13}$ & $(22, 41)$\\
$C_{58, 2}$ & $( 2, 41)$ & $C_{58, 8}$ & $(12, 36)$ & $C_{58,14}$ & $(22, 48)$ \\
$C_{58, 3}$ & $( 2, 58)$ & $C_{58, 9}$ & $(12, 53)$ & $C_{58,15}$ & $(22, 58)$\\
$C_{58, 4}$ & $( 7, 31)$ & $C_{58,10}$ & $(17, 36)$ & $C_{58,16}$ & $(27, 31)$\\
$C_{58, 5}$ & $( 7, 41)$ & $C_{58,11}$ & $(17, 53)$ & $C_{58,17}$ & $(27, 48)$\\
$C_{58, 6}$ & $( 7, 48)$ & $C_{58,12}$ & $(17, 58)$ & $C_{58,18}$ & $(27, 53)$\\
\noalign{\hrule height0.8pt}
\end{tabular}
}
\end{center}
\end{table}

Similar to Sections~\ref{sec:nei} and \ref{sec:d10}, 
we continue the search to find extremal self-dual 
$[58,29,10]$ codes
with weight enumerator for which no extremal
self-dual code was previously known to exist,
by considering self-dual neighbors of $C_{58,i}$ $(i=1,3)$.
These codes $C=D_{58,i}$ are constructed as
\[
\langle (D \cap \langle x \rangle^\perp), x \rangle,
\]
where $D$ and $\supp(x)$  are listed in Table~\ref{Tab:58nei}.
We verified that the codes $D_{58,i}$ have
weight enumerator $W_{58,2}$, where $W$ in Table~\ref{Tab:58nei}
indicates the values $(\beta,\gamma)$ in 
the weight enumerator $W_{58,2}$.
By continuing this process, we found more 
extremal self-dual 
$[58,29,10]$ codes
with weight enumerator for which no extremal
self-dual code was previously known to exist.
The results are listed in  Table~\ref{Tab:58nei}.
From Tables~\ref{Tab:58} and \ref{Tab:58nei},
we have Proposition~\ref{prop:58}.

\begin{table}[thb]
\caption{Extremal self-dual $[58,29,10]$ neighbors}
\label{Tab:58nei}
\begin{center}
{\footnotesize
\begin{tabular}{c|c|l|c}
\noalign{\hrule height0.8pt}
$C$ & $D$ & \multicolumn{1}{c|}{$\supp(x)$} & $W$\\
\hline
$D_{58,1}$& $C_{58,1}$ & $\{3,4,28,30,33,41,43,44,52,53,55,56\}$&$(2,102)$\\
$D_{58,2}$& $C_{58,1}$ & $\{2,3,5,28,33,34,35,41,42,44,45,50\}$&$(2,108)$\\
$D_{58,3}$& $C_{58,3}$ & $\{1,4,6,7,8,39,40,41,42,43,47,52\}$& $(2,28)$\\
\hline
$E_{58,1}$& $D_{58,2}$ & $\{1,3,6,7,34,35,40,47,49,51\}$& $(2,106)$\\
$E_{58,2}$& $D_{58,3}$ & $\{1,6,10,28,31,32,33,40,53,54\} $&$(0,24)$\\
$E_{58,3}$& $D_{58,3}$ & $\{2,6,7,8,31,34,38,45,51,57\}   $&$(1,24)$\\
$E_{58,4}$& $D_{58,3}$ & $\{6,8,28,30,39,40,41,46,57,58\} $&$(1,30)$\\
$E_{58,5}$& $D_{58,3}$ & $\{6,32,33,38,41,42,44,46,47,57\}$&$(2,16)$\\
$E_{58,6}$& $D_{58,3}$ & $\{2,5,6,33,36,39,42,52,53,56\}  $&$(2,20)$\\
$E_{58,7}$& $D_{58,3}$ & $\{5,6,8,36,38,41,48,51,52,55\}  $&$(2,24)$\\
$E_{58,8}$& $D_{58,3}$ & $\{3,6,7,12,32,33,34,43,47,54\}  $&$(2,26)$\\
$E_{58,9}$& $D_{58,3}$ & $\{1,8,13,35,36,44,47,50,53,55\} $&$(2,30)$\\
\hline
$F_{58,1}$ & $E_{58,6}$ & $\{1,4,36,38,39,41,43,45,49,58\}$ &$(0,14)$\\
$F_{58,2}$ & $E_{58,5}$ & $\{1,5,6,7,8,11,29,31,44,46\}$    &$(1,16)$\\
$F_{58,3}$ & $E_{58,5}$ & $\{1,2,30,32,35,44,46,47,53,58\}$ &$(1,18)$\\
$F_{58,4}$ & $E_{58,5}$ & $\{4,6,9,34,41,42,45,50,51,56\}$  &$(1,20)$\\
$F_{58,5}$ & $E_{58,5}$ & $\{4,5,37,41,47,49,50,55,56,58\}$ &$(1,22)$\\
$F_{58,6}$ & $E_{58,5}$ & $\{1,4,6,7,8,35,39,41,42,43\}$    &$(2, 8)$\\
$F_{58,7}$ & $E_{58,5}$ & $\{6,7,12,15,41,43,46,47,49,56\}$ &$(2,12)$\\
$F_{58,8}$ & $E_{58,5}$ & $\{1,6,8,29,34,39,47,50,54,55\}$  &$(2,18)$\\
$F_{58,9}$ & $E_{58,5}$ & $\{3,35,36,38,39,42,47,49,50,56\}$&$(2,22)$\\
\hline
$G_{58,1}$ & $F_{58,1}$ & $\{2,7,11,29,31,32,33,48,49,52\}$&$(0,8)$\\
$G_{58,2}$ & $F_{58,7}$ & $\{3,9,32,38,47,48,49,51,52,55\}$&$(0,4)$\\
$G_{58,3}$ & $F_{58,7}$ & $\{1,8,12,30,33,40,42,49,50,55\}$&$(2,14)$\\
\hline
$H_{58}$ & $G_{58,2}$ & $\{5,6,7,9,32,44,46,47,49,58\}$&$(0,6)$\\
\noalign{\hrule height0.8pt}
\end{tabular}
}
\end{center}
\end{table}

\section{Weight enumerator $W_{58,1}$}
In this section, we give a remark on the possible weight enumerator
$W_{58,1}$.
First, we discuss a general case including $W_{58,1}$.

\begin{prop}\label{prop:restriction}
Let $C$ be a self-dual $[n,n/2,d]$ code
with shadow $S$ of minimum weight $1$.
Let $A_i$ and $B_i$ denote the numbers of vectors of
weight $i$  in $C$ and $S$, respectively.
Suppose that $n \equiv 2 \pmod 8$ and $d \equiv 2 \pmod 4$.
Then $B_{d-1} =A_{d}$.
\end{prop}
\begin{proof}
Let $x$ be the vector of weight $1$ and let
$y$ be a vector of weight $d-1$ in $S$.
Since $x+y \in C$, $x+y$ has weight $d$.
Thus, $B_{d-1} \le A_{d}$.

Now let $c$ be a codeword of weight $d$ in $C$
and let $x$ be the vector of weight $1$ in $S$.
Then we have $x+c \in S$.
From the assumption that $n \equiv 2 \pmod 8$, 
the weight of $x+c$ is congruent to $1 \pmod 4$
by Theorem~5 in~\cite{C-S}.
Hence, from the assumption that $d \equiv 2 \pmod 4$,
$x+c$ has weight $d-1$.
Thus, $B_{d-1} \ge A_{d}$.
The result follows.
\end{proof}

For example, 
Proposition~\ref{prop:restriction}
can be applied to the
following parameters:
\[
(n,d)=(58,10), (74,14) \text{ and } (98,18).
\]

\begin{itemize}
\item 
$(n,d)=(58,10)$:

The possible weight enumerator of the shadow of
an extremal self-dual $[58,29,10]$
code with weight enumerator $W_{58,1}$ is as follows~\cite{C-S}:
\[
y+ \gamma y^9 + (23918-10\gamma)y^{13} + \cdots.
\]
By Proposition~\ref{prop:restriction}, we have
\[
165-2\gamma = \gamma.
\]
Since there is an extremal self-dual $[58,29,10]$
code with weight enumerator $W_{58,1}$ for $\gamma=55$~\cite{Tsai92-2},
we have the following:

\begin{cor}\label{cor}
There is an extremal self-dual $[58,29,10]$
code with weight enumerator $W_{58,1}$ if and only if
$\gamma=55$. 
\end{cor}


\item 
$(n,d)=(74,14)$:

The weight enumerator $W_2$ in \cite[p.~2039]{DGH} is
the possible weight enumerator of an extremal 
self-dual $[74,37,14]$ code with shadow of minimum weight $1$.
By Proposition~\ref{prop:restriction},
we have $\alpha=-135$ for $W_2$ 
in \cite[p.~2039]{DGH}.
The weight enumerators of such a code and its shadow are as follows:
\begin{align*}
&1 + 2044 y^{14} + 159067 y^{16} + 520782 y^{18} + \cdots, \\
&y + 2044 y^{13}  + 679849 y^{17}  + 44010824 y^{21} +\cdots,
\end{align*}
respectively.
It is still unknown whether there is an extremal 
self-dual $[74,37,14]$ code (with shadow of minimum weight $1$).

\item 
$(n,d)=(98,18)$:

The weight enumerator $W_3$ in  \cite[p.~2041]{DGH} is
the unique weight enumerator for an extremal 
self-dual $[98,49,18]$ code with shadow of minimum weight $1$.
The weight enumerators of such a code and its shadow are as follows:
\begin{align*}
&1 + 22116 y^{18} + 2016048 y^{20} + 7181104 y^{22} + \cdots, \\
&y + 22116 y^{17} + 9197152 y^{21} + 964758896 y^{25} + \cdots,
\end{align*}
respectively.
It is still unknown whether there is an extremal 
self-dual $[98,49,18]$ code (with shadow of minimum weight $1$).

\end{itemize}

\bigskip
\noindent
{\bf Acknowledgment.}
This work was supported by JSPS KAKENHI Grant Number 15H03633.
The author would like to thank the anonymous reviewers for 
the useful comments.



\end{document}